\documentclass{amsart}
\usepackage{amssymb}
\newtheorem{thm}{Theorem}[section]

\newtheorem{lem}[thm]{Lemma}

\theoremstyle{definition}

\theoremstyle{remark}

\numberwithin{equation}{section}
\newcommand{\R}{\mathbb R}
\newcommand{\eps}{\varepsilon}

\newcommand{\rt}{\rightarrow}

\begin{document}

\title[Almost-Einstein and nonnegative isotropic curvature]{On the smooth rigidity of almost-Einstein\\ manifolds with
nonnegative isotropic curvature}
\author{harish seshadri}
\address{department of mathematics,
Indian Institute of Science, Bangalore 560012, India}
\email{harish@math.iisc.ernet.in}

%\subjclass{53C21}

%\keywrds{Weyl Curvature, Euler Characteristic,
%Chern-Gauss-Bonnet Theorem, Asymptotically Flat Manifolds, Yamabe
%metric.}
\thanks{This work was supported by DST Grant No. SR/S4/MS-283/05}
%\date{}%
%\dedicatory{}%
%\commby{}%
% ----------------------------------------------------------------
\begin{abstract}

Let $(M^n,g)$, \ $n \ge 4$, be a compact simply-connected
Riemannian manifold with nonnegative isotropic curvature. Given
$0<l\le L$, we prove that there exists $\eps = \eps (l,L,n)$
satisfying the following: If the scalar curvature $s$ of $g$
satisfies
$$   l  \le  s  \le  L $$
and the Einstein tensor satisfies
$$ \vert Ric - \frac {s}{n}g \vert \le \eps$$
then $M$ is diffeomorphic to a symmetric space of compact type.

This is a smooth analogue of the result of S. Brendle that a
compact Einstein manifold with nonnegative isotropic curvature is
isometric to a locally symmetric space.

\end{abstract}
\maketitle
% ----------------------------------------------------------------
\section{Introduction}

A Riemannian manifold $(M,g)$ is said to have nonnegative
isotropic curvature if
$$\ R_{1313} +R_{1414}+R_{2323}+R_{2424}-2R_{1234} \ge 0$$
for every orthonormal 4-frame $\{e_1,e_2,e_3,e_4\}$.

In the case of strict inequality above we say that the manifold
has positive isotropic curvature.  Recently S. Brendle proved that
a compact Einstein manifold with nonnegative isotropic curvature
has to be a locally symmetric space of compact type. In this note
we relax the restriction that the metric is Einstein to the
condition that the Einstein tensor is small in norm and obtain the
following smooth rigidity result:

\begin{thm}\label{ma}
Let $(M^n,g)$, \ $n \ge 4$, be a compact simply-connected
Riemannian manifold with nonnegative isotropic curvature. Given
$0<l \le L$, there exists $\eps = \eps (l,L,n)$ satisfying the
following: If the scalar curvature $s$ of $g$ satisfies
$$     l  \le  s  \le L $$
and the Einstein tensor satisfies
$$ \vert Ric - \frac {s}{n}g \vert \le \eps$$
then $M$ is diffeomorphic to a symmetric space of compact type.

\end{thm}
This result was inspired by the paper of P. Petersen and T. Tao
~\cite{pet} where it is proved that ``almost" quarter-pinching of
sectional curvatures again leads to smooth rigidity as above. The
main difference between their conclusion and ours is that symetric
spaces of rank $\ge 2$ are allowed in our case, while almost
quarter-pinching gives only rank-1 spaces.

We remark that for any $L, \eps$ the conditions $ s \le  L$ and
$\vert Ric - \frac {s}{n}g \vert \le  \eps$ can be achieved just
by rescaling the metric by a large constant. In particular,
consider the connected sum $S^{n-1} \times S^1 \# S^{n-1} \times
S^1$ which admits a metric with positive isotropic curvature by
~\cite{mw}. Rescaling this metric gives the two bounds above.
However this manifold does not support a locally symmetric metric
(irreducible or reducible) of compact type. This is seen by
observing that the fundamental group of the latter space has to
contain an abelian subgroup of finite index. Hence the lower bound
on scalar curvature is necessary. On the other hand it is not
known if just positive Ricci curvature and nonnegative isotropic
curvature already imply that the underlying compact manifold is
diffeomorphic to a locally symmetric space, even without the
assumption of simple-connectivity.

A few remarks about the proof. Let $(M,g)$ be a manifold
satisfying the hypotheses of Theorem \ref{ma}. The main parts of
the proof are obtaining a two-sided bound on sectional curvature
and a lower bound on the injectivity radius of $(M,g)$. An uniform
upper bound on diameter is immediate since the Ricci curvature is
uniformly positive for $\eps$ small enough. The bound on sectional
curvature is the content of Lemma \ref{cur}. The injectivity
radius bound is non-trivial and follows from a theorem of Petrunin
- Tuschmann. To apply their result one needs finite second
homotopy group which is guaranteed by positive isotropic
curvature. To deal with nonnegative isotropic curvature we use the
results of H. Seshadri ~\cite{hs} and S. Brendle ~\cite{bre} which
allow us to reduce the nonnegative case to the positive case.

\section{proof of Theorem \ref{ma}}
We begin with a simple but useful lemma. Let $c \in \R$. By
$$K^{iso} \ge c$$ we mean that
$$K^{iso}(e_i,e_j,e_k,e_l):=\ R_{ikik} +R_{ilil}+R_{jkjk}+R_{jljl}-2R_{ijkl} \ge c$$
for every orthonormal 4-frame $\{e_i,e_j,e_k,e_l\}$.

\begin{lem}\label{cur}
Given $c, C \in \R$, there exists $b=b(c,C,n)$ such that if
$(M^n,g)$ is a Riemannian manifold with
$$K^{iso} \ge c, \ \ \ \  s \le C,$$
then the norm of the Weyl tensor $W$ is bounded by $b$:
$$ \vert W \vert \le b.$$
\end{lem}

\begin{proof}
The proof is similar to that of Proposition 2.5 of ~\cite{mw}.
Note that
$$K^{iso}(e_i,e_j,e_k,e_l)+K^{iso}(e_i,e_j,e_l,e_k): =  2( R_{ikik} +R_{ilil}+R_{jkjk}+R_{jljl}).$$
From this it follows that $4s$ can be expressed as a sum of
$n(n-1)$  isotropic curvatures, Since we have an upper bound on
$s$ and a lower bound on $K^{iso}$, we get an upper bound
$b_1=b_1(c,C,n)$ for $K^{iso}$. We have a lower bound on $K^{iso}$
by hypothesis and hence we have two-sided bounds on
$$4W_{ijkl} = 4R_{ijkl} = K^{iso}(e_i,e_j,e_l,e_k) - K^{iso}(e_i,e_j,e_k,e_l)$$
depending only on $c,C$ and $n$. Since this holds for an arbitrary
orthonormal 4-frame, we can apply the above bound to the 4-frame
$$\left \{ e_i, \ \frac {1}{\sqrt 2} (e_j -e_l), \ e_k, \  \frac {1}{\sqrt 2} (e_j + e_l)  \right \}$$
to see that $\vert W_{ijkj}-W_{ilkl} \vert \le b_2(c,C,n)$. Since
$\sum_p W_{ipkp} =0$, this implies that $\vert W_{ipkp} \vert \le
b_3(c,C,n)$. Hence $\vert W \vert \le b_4(c,C,n).$

%Applying the argument to the 4-frames $\{e_1,e_2,e_3,e_i\}$ for $
%5 \le i \le n$ we get $\vert R_{1232}-R_{1i3i} \vert \le
%b_2(c,C,n)$. This also holds for $i=4$ as seen earlier and for
%$i=2$ trivially. Adding these inequalities and applying the
%triangle inequality we get
%$$ \vert (n-2)R_{1232} - Ric(e_1,e_3) \vert \le b_3(c,C,n).$$
%The bound on Ricci curvature gives
%$$\vert R_{1232} \vert \le b_4(c,C,n).$$
%The fact that $\vert R_{1234} \vert$ and $\vert R_{1232} \vert$
%are bounded for an arbitrary 4-frame implies that the norm of the
%Weyl tensor of $g$ is bounded. Since the norm of the Ricci tensor
%is bounded by assumption, we see that the norm of $R$ is bounded.

\end{proof}
An immediate corollary of Lemma \ref{cur} is that {\it an upper
bound on scalar curvature, a lower bound on isotropic curvature
and an upper bound on the norm of the Ricci tensor gives a bound
on the norm of the Riemann curvature tensor}. This applies, in
particular, to a metric satisfying the hypotheses of Theorem
\ref{ma}.  \vspace{3mm}

The first restriction we impose on $\eps$ is that $\eps \le  \frac
{l}{2n}$. This implies that the Ricci curvature is uniformly
positive:
\begin{equation}\label{ric}
Ric \ge  \frac {l}{2n} g.
\end{equation}

We claim that we can find such an $\eps$ if we assume  that
$(M,g)$ has positive isotropic curvature. Then, by Micallef-Moore
~\cite{mm}, $\pi_2(M)=0$. Theorem 0.4 of ~\cite{pet} states that the
injectivity radius of a compact simply-connected Riemannian
$n$-manifold with finite second homotopy group, bounded sectional
curvature $\vert K \vert \le a$ and positive Ricci curvature $Ric
> bg$ has a positive lower bound on injectivity radius dependent
only on $a,b$ and $n$. The comment following Lemma \ref{cur} and
(\ref{ric}) give us the required bounds on curvature.
\vspace{2mm}

{\it Remark}: If $M$ is even-dimensional, then one has the
following alternative proof for a lower bound on injectivity
radius $inj$. If $inj \rt 0$, then all the characteristic numbers
of $M$, in particular the Euler characteristic , would have to
vanish. On the other hand, a simply-connected Riemannian
$n$-manifold with positive isotropic curvature has to be
homeomorphic to the $n$-sphere ~\cite{mm}. This contradiction
shows that collapse cannot occur in even-dimensions. \vspace{2mm}

Now suppose that there is no $\eps$ for which the conclusion of
Theorem \ref{ma} holds. Then we get a sequence of $(M_i,g_i)$ of
Riemannian $n$-manifolds, none of which is diffeomorphic to a
symmetric space of compact type, with uniformly bounded sectional
curvatures and diameter (by Myers-Bonnet, since (\ref{ric}) holds)
and injectivity radius bounded below. As in ~\cite{pt} we can
assume that a subsequence converges in the $C^\infty$ topology to
a smooth complete Riemannian manifold $(M,g)$. This manifold will
have to be Einstein, of finite diameter (hence compact) and of
nonnegative isotropic curvature. By ~\cite{bre}, $(M,g)$ is
isometric to a symmetric space of compact type or flat. Since
$M_i$ is diffeomorphic to $M$ for large $i$ and $M_i$ is
simply-connected, $M$ cannot be flat. Hence $M_i$ is diffeomorphic
to a symmetric space of compact type for large $i$, which is a
contradiction.

Hence we have established the existence of
\begin{equation}\label{pos}
\eps_p =\eps_p(l,L,n)
\end{equation}
which yields the conclusion in the presence of positive isotropic
curvature.

Next consider the general case of nonnegative isotropic curvature.

\begin{lem}\label{com}
Let $(M^n,\ g)$ be a compact simply-connected Riemannian manifold
with nonnegative isotropic curvature. Suppose that
$$    0 < l  \le   s_h  \le  L, \ \ \vert Ric_g - \frac {s_g}{n}g \vert_g \le \eps$$
for some $0<l \le  L$ and $\eps \le  \frac {2l}{n} $.

Let $(N^k,h)$ be an irreducible factor in the de Rham
decomposition of $N$. If $k=2$ or $3$, $N$ is diffeomorphic to
$S^2$ or $S^3$. If $k \ge 4$, then $(N,h)$ has nonnegative
isotropic curvature and

$$ 0 < \frac {2l}{n} < s_g \le  L , \ \ \ \ \vert Ric_h - \frac {s_h}{k}h \vert_g \le \eps .$$

\end{lem}
\begin{proof}
The statement about $k=2$ or $3$ follows from the description of
reducible manifolds with nonnegative isotropic curvature given by
M. Micallef and M. Wang (Theorem 3.1, ~\cite{mw}). If $k \ge 4$,
note that
\begin{align} \notag
\vert Ric_g - \frac {s_g}{n}g \vert_g^2 & \ \ \ge \ \ \vert Ric_h - \frac {s_g}{n}h \vert_h ^2 \notag \\
& \ = \ \ \vert Ric_h - \frac {s_h}{k} \vert_h ^2 \ \ + \ \ k^2 \vert
\frac{s_h}{k} - \frac {s_g}{n} \vert ^2. \notag
\end{align}

Hence
$$ \vert Ric_h - \frac {s_h}{k}h \vert_h  \ \le \ \eps .$$
and
$$s_h \ \ge \ \frac {k}{n} s_g - \eps \ \ge \ \frac {4l}{n} -   \frac {2l}{n} \ = \ \frac {2l}{n} $$

Moreover, since $(M,g)$ has positive Ricci curvature, so does each irreducible component and
hence $s_h \le s_g  \le L$.
\end{proof}

We can now complete the proof of the theorem by induction. The
proof for the first nontrivial dimension $n=4$ is the same as that
for the inductive step, so we assume that the result is true in
all dimensions less than $n$. Let $(M^n,g)$ be a manifold as in
Theorem \ref{ma} with the norm of the Einstein tensor being
smaller than
$$\eps_r(l,L,n) := min \left \{\frac {2l}{n}, \ \eps(\frac {2l}{n},L,4),...,\eps(\frac {2l}{n},L,n-1) \right \}.$$

If $(M,g)$ is reducible, it is enough to prove that each irreducible component of $(M,g)$ is
diffeomorphic to a symmetric space of compact type. Let $(N^k,h), \ 1 \le k \le n-1$ be such a component.
By Lemma \ref{com} and the inductive hypothesis we are done.

Suppose $(M,g)$ is irreducible. We claim that if the Einstein
tensor of $g$ is $\frac {1}{2} {\eps_p(\frac {l}{2},2L,n)}$-small
(where $\eps_p$ is defined by (\ref{pos})) then we have the
desired conclusion. By the results of ~\cite{hs} and ~\cite{bre}
the following holds: Either $(M,g)$ is diffeomorphic to a
symmetric space with nonconstant sectional curvature or we can
find a metric $\bar g$ with positive isotropic curvature as close
(in the $C^\infty$ topology) to $g$ as we want. Choose $\bar g$ so
close to $g$ that
$$    0 < \frac {l}{2}  \le   s_{\bar g}  \le  2L, \ \ \vert Ric_{\bar g} -
\frac {s_{\bar g}}{n}{\bar g} \vert_{\bar g} \le \eps_p.$$

Since $(M,\bar g)$ has positive isotropic curvature and satisfies
the above bounds, it is diffeomorphic to a symmetric space by our
earlier result.

Finally we choose
$$\eps (l,L,n) = min \left \{ \frac {1}{2} \ {\eps_p ( \frac
{l}{2},2L,n )},\ \ \eps_r(l,L,n) \right \}.$$

\hfill $\square$

\end{document}